\documentclass[12pt, journal, draftclsnofoot, onecolumn]{IEEEtran}
\ifCLASSOPTIONcompsoc
\else
\fi

\ifCLASSINFOpdf
\else
\fi
\usepackage[parfill]{parskip}    
\usepackage{amsmath,amssymb,latexsym,enumerate, mathrsfs}
\usepackage{graphicx}
\usepackage{hyperref}

\usepackage{graphicx}
\usepackage{array}
\usepackage{times}
\usepackage{subfig}

\usepackage{moreverb}
\usepackage{marvosym}
\usepackage{color,soul}
\definecolor{lightblue}{rgb}{.90,.95,1}
\sethlcolor{yellow}

\usepackage{hyperref}

\newtheorem{lemma}{\bf Lemma}
\newtheorem{proposition}{\bf Proposition}

\newtheorem{remark}{Remark}

\newtheorem{assumption}{\bf Assumption}

\begin{document}
%
\title{Exit Probabilities for a Chain of Distributed Control Systems with Small Random Perturbations}

\author{Getachew~K.~Befekadu,~\IEEEmembership{Member,~IEEE} and~Panos~J.~Antsaklis,~\IEEEmembership{Fellow,~IEEE}
\thanks{This work was supported in part by the National Science Foundation under Grant No. CNS-1035655. The first author acknowledges support from the Department of Electrical Engineering, University of Notre Dame.}
\IEEEcompsocitemizethanks{\IEEEcompsocthanksitem G. K. Befekadu is with the Department
of Electrical Engineering, University of Notre Dame, Notre Dame, IN 46556, USA.\protect\\
E-mail: gbefekadu1@nd.edu
\IEEEcompsocthanksitem P. J. Antsaklis is with the Department
of Electrical Engineering, University of Notre Dame, Notre Dame, IN 46556, USA.\protect\\
E-mail: antsaklis.1@nd.edu}}

\markboth{IEEE TRANSACTIONS ON AUTOMATIC CONTROL,~Vol.~xx, No.~xx, xxx~xxx}%
{Shell \MakeLowercase{\textit{et al.}}: Bare Advanced Demo of IEEEtran.cls for Journals}
\IEEEcompsoctitleabstractindextext{%
\begin{abstract}
In this paper, we consider a diffusion process pertaining to a chain of distributed control systems with small random perturbation. The distributed control system is formed by $n$ subsystems that satisfy an appropriate H\"{o}rmander condition, i.e., the second subsystem assumes the random perturbation entered into the first subsystem, the third subsystem assumes the random perturbation entered into the first subsystem then was transmitted to the second subsystem and so on, such that the random perturbation propagates through the entire distributed control system. Note that the random perturbation enters only in one of the subsystems and, hence, the diffusion process is degenerate, in the sense that the backward operator associated with it is a degenerate parabolic equation. Our interest is to estimate the exit probability with which a diffusion process (corresponding to a particular subsystem) exits from a given bounded open domain during a certain time interval. The method for such an estimate basically relies on the interpretation of the exit probability function as a value function for a family of stochastic control problems that are associated with the underlying chain of distributed control systems.
\end{abstract}

}

\maketitle

\IEEEdisplaynotcompsoctitleabstractindextext

%
\IEEEpeerreviewmaketitle

\section{Introduction}	 \label{S1}
In this paper, we consider the problem of estimating probabilities with which the diffusion process $x^{\epsilon, i}(t)$, for $i = 1, 2, \ldots, n$, exits from a given bounded open domain during a certain time interval pertaining to the following $n$ distributed control systems (see Fig.~\ref{Fig-DCS})
\begin{align}
\left.\begin{array}{l}
d x^{\epsilon,1}(t) = f_1\bigl(t, x^{\epsilon,1}(t), u_1(t)\bigr) dt + \sqrt{\epsilon} \sigma\bigl(t, x^{\epsilon,1}(t))dW(t) \\
d x^{\epsilon,2}(t) = f_2\bigl(t, x^{\epsilon,1}(t), x^{\epsilon,2}(t), u_2(t)\bigr) dt  \\
 \quad\quad\quad~ \vdots  \\
d x^{\epsilon,i}(t) = f_{i}\bigl(t, x^{\epsilon,1}(t), x^{\epsilon,2}(t), \ldots, x^{\epsilon,i}(t), u_{i}(t)\bigr) dt  \\
 \quad\quad\quad~ \vdots  \\
d x^{\epsilon,n}(t) = f_n\bigl(t, x^{\epsilon,1}(t), x^{\epsilon,2}(t), \ldots, x^{\epsilon,n}(t), u_{n}(t)\bigr) dt\\
 x^{\epsilon,1}(s)=x_s^{\epsilon,1}, \,\, x^{\epsilon,2}(s)=x_s^{\epsilon,2}, \,\, \ldots, \,\, x^{\epsilon,n}(s)=x_s^{\epsilon,n}, \,\, s \le t \le T
\end{array}\right\}  \label{Eq1} 
\end{align}
where
\begin{itemize}
\item[-] $x^{\epsilon,i}(\cdot)$ is an $\mathbb{R}^{d}$-valued diffusion process that corresponds to the $i$th-subsystem,
\item[-] the functions $f_i \colon [0, \infty) \times \mathbb{R}^{i \times d} \times \mathcal{U}_i \rightarrow \mathbb{R}^{d}$ are uniformly Lipschitz, with bounded first derivatives, $\epsilon$ is a small positive number (which represents the level of random perturbation in the system),
\item[-] $\sigma \colon [0, \infty) \times \mathbb{R}^{d} \rightarrow \mathbb{R}^{d \times m}$ is Lipschitz with the least eigenvalue of $\sigma(\cdot, \cdot)\,\sigma^T(\cdot, \cdot)$ uniformly bounded away from zero, i.e., 
\begin{align*}
 \sigma(t, x)\,\sigma^T(t, x) \ge \lambda I_{d \times d} , \quad \forall x \in \mathbb{R}^{d},
\end{align*}
for some $\lambda > 0$,
\item[-] $W(\cdot)$ is a $d$-dimensional standard Wiener process (with $W(0) = 0$),
\item[-] $u_i(\cdot)$ is a $\,\mathcal{U}_i$-valued measurable control process to the $i$th-subsystem (i.e., an admissible control from the measurable set $\mathcal{U}_i\subset \mathbb{R}^{r_i}$) such that for all $t > s$, $W(t)-W(s)$ is independent of $u_i(\nu)$ for $\nu \le s$ (nonanticipativity condition) and
\begin{align*}
\mathbb{E} \int_{s}^{t_1} \vert u_i(t)\vert^2 dt < \infty, \quad \forall t_1 \ge s,
\end{align*}
\end{itemize}
for $i = 1, 2, \ldots, n$.

\begin{figure}[tbh]
\vspace{-5 mm}
\begin{center}
\subfloat{\includegraphics[width=50mm]{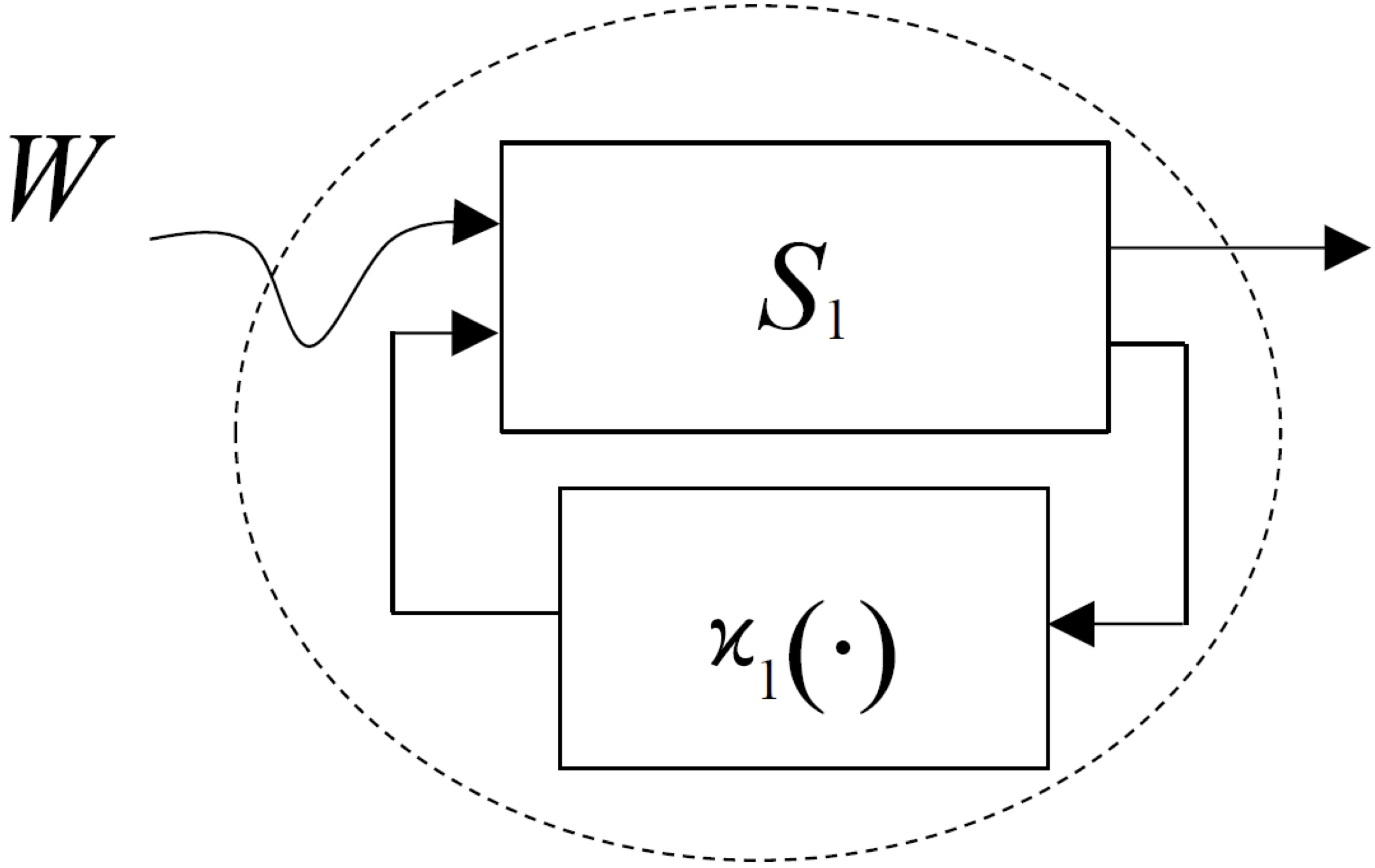}}
\subfloat{\includegraphics[width=62mm]{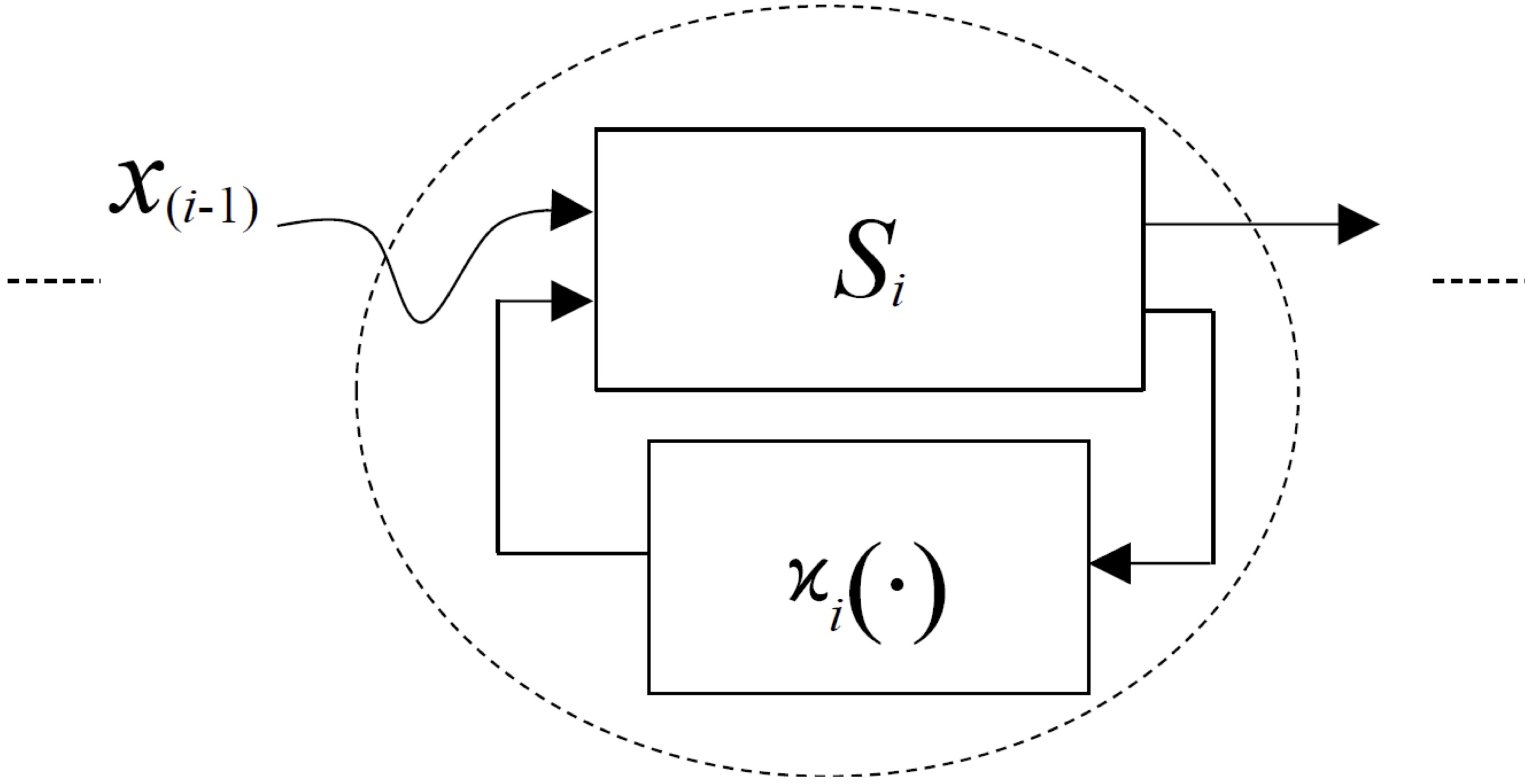}}
\subfloat{\includegraphics[width=51mm]{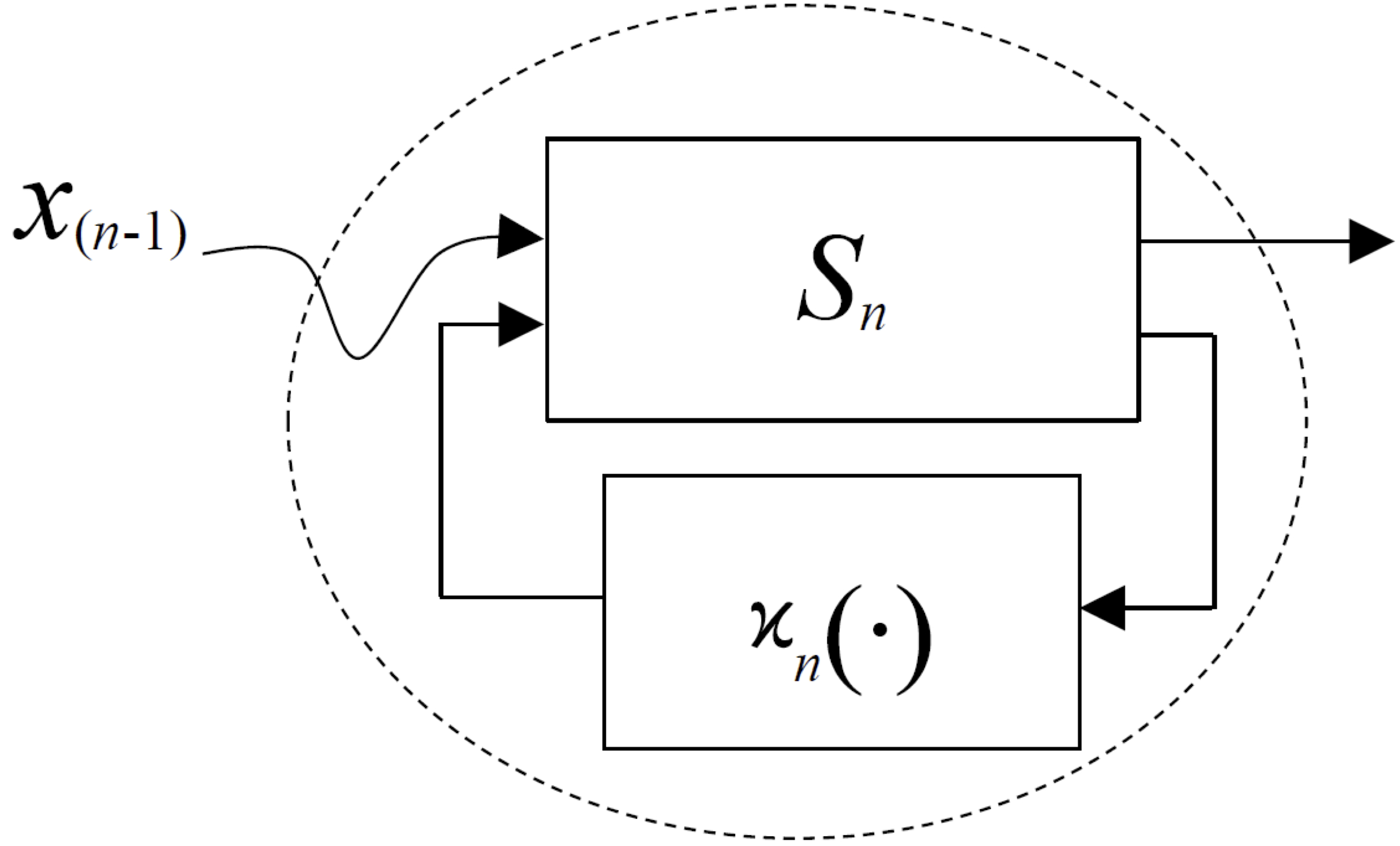}}\\
{$ \begin{array}{r@{\ }c@{\ }l}
~\vspace{-5 mm}\\
\text{where} & \\
&S_1:  \, \, d x^{\epsilon,1}(t) = f_1\bigl(t, x^{\epsilon,1}(t), u_1(t)\bigr) dt + \sqrt{\epsilon} \sigma\bigl(t, x^{\epsilon,1}(t))dW(t), \\
&S_i: \,\, d x^{\epsilon,i}(t) = f_{i}\bigl(t, x^{\epsilon,1}(t), x^{\epsilon,2}(t), \ldots, x^{\epsilon,i}(t), u_{i}(t)\bigr) dt, ~ i = 2, 3, \ldots n,\\
     & \quad \quad u_j(t)= \kappa_j(x^{\epsilon,j}(t)), ~ t \in [s, T], ~ j =1, 2,  \ldots n.
\end{array}$}
\caption{A chain of distributed control systems with small random perturbations} \label{Fig-DCS}
\vspace{-5 mm}
\end{center}
\end{figure}

In what follows, we consider a particular class of admissible controls $u_{i}(\cdot) \in \mathcal{U}_{i}$, for $i = 1, 2 \ldots, n$, of the form $\kappa_{i}(x^{\epsilon,i}(t))$, $\forall t \ge s$, with a measurable map $\kappa_{i}$ from $\mathbb{R}^{d}$ to $\mathcal{U}_{i}$ and, thus, such a measurable map $\kappa_i$ is called a stationary Markov control. 

\begin{remark} \label{R1}
Note that, in Equation~\eqref{Eq1}, the function $f_i$, with the admissible control $\kappa_i(x^{\epsilon,i})$, depends only on $x^{\epsilon,1}(t)$, $x^{\epsilon,2}(t)$, \ldots, $x^{\epsilon,i}(t)$ and satisfies an appropriate H\"{o}rmander condition (e.g., see \cite{Hor67} for further discussion). Furthermore, the random perturbation has to pass through the second subsystem, the third subsystem, \dots, and the $(i-1)$th-subsystem to reach for the $i$th-subsystem. Hence, such a chain of distributed control systems is described by an $n \times d$-dimensional diffusion process, which is degenerate in the sense that the backward operator associated with it is a degenerate parabolic equation. 
\end{remark}

Let $D \subset \mathbb{R}^{d}$ be a bounded open domain with smooth boundary (i.e., $\partial D$ is a manifold of class $C^2$). Moreover, let $\tau_D^{\epsilon,\ell}=\tau_D^{\epsilon,{\ell}}(s, x^{\epsilon,1}, \ldots, x^{\epsilon,\ell})$, for $\ell = 2, 3, \ldots, n$, be the first exit-time for the diffusion process $x^{\epsilon, \ell}(t)$ (corresponding to the $\ell$th-subsystem, with the admissible control $\kappa_{\ell}(\xi^{\ell}(t))$, for $t \ge s \ge 0$) from the given domain $D$, i.e., 
\begin{align}
\tau_D^{\epsilon,\ell} = \inf \Bigl\{ t > s \, \bigl\vert \, x^{\epsilon,\ell}(t) \in \partial D \Bigr\}, \label{Eq2}
\end{align}
which depends on the behavior of the solutions to the following (deterministic) chain of distributed control systems, i.e.,
\begin{align}
 d \xi^{j}(t) = \hat{f}_{j}(t, \xi^{1}(t), \ldots, \xi^{j}(t)\bigr) dt, \quad \xi^{j}(s)=x_s^{j}, \quad s \le t \le T, \label{Eq3}
\end{align}
for $j = 1, \ldots, \ell$, with
\begin{align*}
\hat{f}_{j}(t, \xi^{1}(t), \ldots, \xi^{j}(t)) \triangleq f_{j}(t, \xi^{1}(t), \ldots, \xi^{j}(t), \kappa_{j}(\xi^{j}(t))).
\end{align*} 

For a fixed (given) $T$, let us define the exit probability as   
\begin{align}
q^{\epsilon,\ell} \bigl(s, x^{\epsilon,1}, \ldots, x^{\epsilon,\ell}\bigr) = \mathbb{P}_{s, x_{\widehat{1,\ell}}}^{\epsilon} \Bigl\{ \tau_D^{\epsilon,\ell} \le T \Bigr\}, \quad \ell = 2, 3, \ldots, n, \label{Eq4}
\end{align}
where such a probability $\mathbb{P}_{s, x_{\widehat{1,\ell}}}^{\epsilon} \bigl\{\cdot\bigr\}$ is conditioned on the initial points $x_s^{\epsilon,j} \in \mathbb{R}^d$, for $j=1, \ldots \ell$, as well as on the class of admissible controls.\footnote{$\mathbb{P}_{s, x_{\widehat{1,\ell}}}^{\epsilon} \bigl\{\cdot \bigr\} \triangleq \mathbb{P}_{s, x^{\epsilon,1}, \ldots, x^{\epsilon,\ell}}^{\epsilon} \bigl\{\cdot\bigr\}$}

Notice that the backward operator for the diffusion process $\bigl(x^{\epsilon,1}(t), \ldots, x^{\epsilon,\ell}(t)\bigr)$, with $\kappa_{j}(x^{\epsilon,j}(t))$, for $j = 1, \ldots, \ell$ and $\forall t \ge s$, when applied to a certain function $\upsilon^{\epsilon,\ell}\bigl(s, x^{\epsilon,1}, \ldots, x^{\epsilon,\ell}\bigr)$, is given by
\begin{align}
 \upsilon_s^{\epsilon,\ell} + \mathcal{L}^{\epsilon, \ell} \upsilon^{\epsilon,\ell} \triangleq \upsilon_s^{\epsilon,\ell} + \frac{\epsilon}{2} \operatorname{tr}\Bigl \{a\, \upsilon_{x^{\epsilon,1}x^{\epsilon,1}}^{\epsilon,\ell} \Bigr\} + \sum\nolimits_{j=1}^\ell \hat{f}_{j}^T \upsilon_{x^{\epsilon,j}}^{\epsilon,\ell},  \label{Eq5}
\end{align}
for $j = 1, \ldots, \ell$, with $a(s, x^{\epsilon,1})=\sigma(s, x^{\epsilon,1})\,\sigma^T(s, x^{\epsilon,1})$ and 
\begin{align*}
\hat{f}_{j}\bigl(t, x^{\epsilon,1}(t), \ldots, x^{\epsilon,j}(t)) = f_{j}\bigl(t, x^{\epsilon,1}(t), \ldots, x^{\epsilon,j}(t), \kappa_{j}(x^{\epsilon,j}(t))).
\end{align*}

Let $\Omega_{\ell}$ be the open set
\begin{align*}
\Omega_{\ell} = (0, T) \times \mathbb{R}^{(\ell-1) \times d} \times D, \quad \ell =2, 3, \ldots, n.
\end{align*}
Further, let us denote by $C^{\infty}(\Omega_{\ell})$ the spaces of infinitely differentiable functions on $\Omega_{\ell}$, and by $C_0^{\infty}(\Omega_{\ell})$ the space of the functions $\phi \in C^{\infty}(\Omega_{\ell})$ with compact support in $\Omega_{\ell}$. A locally square integrable function $ \upsilon^{\epsilon,\ell}\bigl(s, x^{\epsilon,1}, \ldots, x^{\epsilon,\ell}\bigr)$ on $\Omega_{\ell}$ is said to be a probabilistic solution to the following equation
\begin{align}
 \upsilon_s^{\epsilon,\ell} + \mathcal{L}^{\epsilon, \ell} \upsilon^{\epsilon,\ell} = 0,  \label{Eq6}
\end{align}
if, for any test function $\phi \in C_0^{\infty}(\Omega_{\ell})$, the following holds true
\begin{align}
 \int_{\Omega_{\ell}} \Bigl(-\phi_t + {\mathcal{L}^{\epsilon, \ell}}^{\ast} \phi \Bigr)\upsilon^{\epsilon,\ell} d \Omega_{\ell} = 0, \quad \ell = 2, 3, \ldots, n, \label{Eq7}
\end{align}
where $d \Omega_{\ell}$ denotes the Lebesgue measure on $\mathbb{R}^{(\ell \times d) + 1}$ and ${\mathcal{L}^{\epsilon, \ell}}^{\ast}$ is an adjoint operator corresponding to $\mathcal{L}^{\epsilon, \ell}$
\begin{align}
 {\mathcal{L}^{\epsilon, \ell}}^{\ast} \phi = \frac{\epsilon}{2} \sum\nolimits_{j=1}^d \sum\nolimits_{m=1}^d \bigr(a_{j,m} \phi \bigr)_{x_j^{\epsilon,1} x_m^{\epsilon,1}} - \sum\nolimits_{j=1}^\ell \sum\nolimits_{m=1}^d \bigl(\hat{f}_{j}\phi\bigr)_{x_m^{\epsilon,j}}.  \label{Eq8}
\end{align}

In this paper, we assume that the following statements hold for the chain of distributed control systems in Equation~\eqref{Eq1}.
\begin{assumption} \label{AS1} ~\\\vspace{-10mm}
\begin{enumerate} [(a)]
\item The functions $f_{i}$, for $i = 1, 2, \ldots, n$, are bounded $C^{\infty}(\Omega_{0,i})$-functions, with bounded first derivatives, where $\Omega_{0,i} = (0, \infty) \times \mathbb{R}^{(i \times d)}$. Moreover, $\sigma$ and $\sigma^{-1}$ are bounded $C^{\infty} \bigl((0, \infty) \times \mathbb{R}\bigr)$-functions, with bounded first derivatives.
\item The backward operator in Equation~\eqref{Eq5} is hypoelliptic in $C^{\infty}(\Omega_{0,\ell})$, for each $\ell =2, 3, \ldots, n$ (e.g., see \cite{Hor67} or \cite{Ell73}).
\item For each $\ell = 2, 3, \ldots, n$, let $n(x^{\epsilon,\ell})$ be the outer normal vector to $\partial D$ and, further, let $\Gamma_{\ell}^{+}$ and $\Gamma_{\ell}^{0}$ denote the sets of points $(t, x^{\epsilon,1}, \ldots, x^{\epsilon,\ell})$, with $x^{\epsilon,\ell} \in \partial D$, such that
\begin{align*}
\hat{f}_{\ell}^T (t, x^{\epsilon,1}, \ldots, x^{\epsilon,\ell}) n(x^{\epsilon,\ell})
\end{align*}
is positive and zero, respectively.\footnote{Here, we remark that
\begin{align*}
\mathbb{P}_{s, x_{\widehat{1,\ell}}}^{\epsilon} \Bigl\{ \Bigl( \tau_D^{\epsilon,\ell}, x^{\epsilon,1}, \ldots, x^{\epsilon,\ell}\Bigr) \in \Gamma_{\ell}^{+} \bigcup \Gamma_{\ell}^{0},\,\, \tau_D^{\epsilon,\ell} < \infty \Bigr \} =1, \quad \forall \bigl(s, x^{\epsilon,1}(s), \ldots, x^{\epsilon,\ell}(s)\bigr) \in \Omega_{0,\ell}, \\ \ell =2, 3, \ldots, n.
\end{align*}
Notice that if
\begin{align*}
\mathbb{P}_{s, x_{\widehat{1,\ell}}}^{\epsilon} \Bigl\{ \Bigl(t, x^{\epsilon,1}, \ldots, x^{\epsilon,\ell}\Bigr) \in \Gamma_{\ell}^{0} ~ \text{for some} ~ t \in [s, T] \Bigr \} = 0, \quad \forall \bigl(s, x^{\epsilon,1}(s), \ldots, x^{\epsilon,\ell}(s)\bigr) \in \Omega_{0,\ell},
\end{align*}
and, moreover, if $\tau_D^{\epsilon,\ell} \le T$, then we have $\Bigl( \tau_D^{\epsilon,\ell}, x^{\epsilon,1}(\tau_D^{\epsilon,\ell}), \ldots, x^{\epsilon,\ell}(\tau_D^{\epsilon,\ell}) \Bigr) \in \Gamma^{+}$ almost surely (see \cite[Section~7]{StrVa75}).}
\end{enumerate}
\end{assumption}

\begin{remark}  \label{R2}
Note that, from Assumptions~\ref{AS1}(a)-(b), each matrix $\Bigr(\frac{\partial {f_{\ell}}_i}{\partial x_j^{\epsilon,1}}\Bigl)_{ij}$, for $\ell = 2, 3, \ldots n$, has full rank $d$ everywhere in $\Omega_{0,\ell}$, since the backward operator in Equation~\eqref{Eq5} is hypoelliptic. In particular, the hypoellipticity assumption is related to a strong accessibility property of controllable nonlinear systems that are driven by white noise (e.g., see \cite[Section~3]{Ell73} for further discussion). That is, the hypoellipticity assumption implies that the diffusion process $x^{\epsilon,\ell}(t)$ has a transition probability density $p\bigl(t,(x^{\epsilon,1}, \ldots, x^{\epsilon,\ell}); \mu\bigr)$, which is $C^{\infty}$ on $(0, \infty) \times \mathbb{R}^{2(d \times \ell)}$, and which also satisfies the forward equation $p_t = {\mathcal{L}^{\epsilon, \ell}}^{\ast} p$ (in the variables $(t, \mu))$.
\end{remark}

In Section~\ref{S2}, we present our main result -- where, using the Ventcel-Freidlin asymptotic estimates \cite{VenFre70} (cf. \cite[Chapter~14]{Fre76} or \cite{FreWe84}) and the stochastic control arguments from Fleming \cite{Fle78}, we provide an asymptotic bound on the exit probability $q^{\epsilon,\ell} \bigl(s, x^{\epsilon,1}, \ldots, x^{\epsilon,\ell} \bigr)$, i.e.,
\begin{align*}
 I^{\epsilon,\ell} \bigl(s, x^{\epsilon,1}, \ldots, x^{\epsilon,\ell} \bigr) \rightarrow I^{0,\ell} \bigl(s, x^{\epsilon,1}, \ldots, x^{\epsilon,\ell} \bigr) \quad \text{as} \quad \epsilon \rightarrow 0,
\end{align*}
where
\begin{align*}
I^{\epsilon,\ell} \bigl(s, x^{\epsilon,1}, \ldots, x^{\epsilon,\ell} \bigr) = -\epsilon \log q^{\epsilon,\ell} \bigl(s, x^{\epsilon,1}, \ldots, x^{\epsilon,\ell} \bigr), \quad \ell = 2, 3, \ldots, n.
\end{align*}
Such an asymptotic estimate for $I^{\epsilon,\ell} \bigl(s, x^{\epsilon,1}, \ldots, x^{\epsilon,\ell} \bigr)$ relies on the interpretation of the exit probability function as a value function for a family of stochastic control problems that can be associated with the underlying chain of distributed control systems. Finally, we provide concluding remarks in Section~\ref{S3}.

Before concluding this section, it is worth mentioning that some interesting studies on the asymptotic behavior of exit probabilities for dynamical systems with small random perturbations have been reported in literature (for example, see \cite{Hern81}, \cite{DelM10} or \cite{StrVa75} in the context of estimating density functions for degenerate diffusions; see \cite{She91} or \cite{Day87} in the context of nondegenerate diffusions; and see also \cite{BisB11} in the context of exit-time and invariant measure for small noise constrained diffusions).

\section{Main Results} \label{S2}
\subsection{The exit probabilities} \label{S2(1)}
Let $(x^{\epsilon,1}(t), x^{\epsilon,2}(t), \ldots, x^{\epsilon,\ell}(t))$, with $\ell = 2, 3, \ldots, n$, for $0 \le t \le T$, be the diffusion process. Further, let us consider the following boundary value problem
\begin{align}
\left.\begin{array}{c}
\upsilon_s^{\epsilon,\ell} + \mathcal{L}^{\epsilon, \ell} \upsilon^{\epsilon,\ell} = 0 \quad \text{in} \quad \Omega_{\ell} = (0, T) \times \mathbb{R}^{(\ell-1) \times d} \times D  \\
 \upsilon^{\epsilon,\ell} \bigl(s, x^{\epsilon,1}, \ldots, x^{\epsilon,\ell} \bigr) = 1 \quad \text{on} \quad \Gamma_{T, \ell}^{+}  \\
 \upsilon^{\epsilon,\ell} \bigl(s, x^{\epsilon,1}, \ldots, x^{\epsilon,\ell} \bigr) = 0 \quad \text{on} \quad \{T\} \times \mathbb{R}^{(\ell-1) \times d} \times D
\end{array}\right\}   \label{Eq9}
\end{align}
where $\mathcal{L}^{\epsilon, \ell}$ is the backward operator in Equation~\eqref{Eq5} and
\begin{align*}
 \Gamma_{T, \ell}^{+} = \Bigl\{\bigl(s, x^{\epsilon,1}, \ldots, x^{\epsilon,\ell} \bigr) \in \Gamma_{\ell}^{+}\, \bigl \vert \, 0 < s \le T \Bigr\}.  
\end{align*}
Let $\Omega_{0T,\ell}$ be the set consisting of $\Omega_{\ell} \bigcup \{T\} \times \mathbb{R}^{(\ell-1) \times d} \times D$, together with the boundary points $\bigl(s, x^{\epsilon,1}, \ldots, x^{\epsilon,\ell} \bigr) \in \Gamma_{\ell}^{+}$, with $0< s<T$. Then, the following proposition provides a solution to the exit probability $\mathbb{P}_{s, x_{\widehat{1,\ell}}}^{\epsilon} \Bigl\{ \tau_D^{\epsilon,\ell} \le T \Bigr\}$, for each $\ell = 2, 3, \ldots, n$, with which the diffusion process $x^{\epsilon,\ell}(t)$ exits from the domain $D$.
 
\begin{proposition} \label{P1}
Suppose that the statements~(a)-(c) in the above assumption (i.e., Assumption~\ref{AS1}) hold true. Then, the exit probability $q^{\epsilon,\ell} (s, x^{\epsilon,1}, \ldots, x^{\epsilon,\ell}) = \mathbb{P}_{s, x_{\widehat{1,\ell}}}^{\epsilon} \bigl\{ \tau_D^{\epsilon,\ell} \le T \bigr\}$ is a smooth solution to the boundary value problem in Equation~\eqref{Eq9}, and, moreover, it is a continuous function on $\Omega_{0T,\ell}$.
\end{proposition}
In order to prove the above proposition, we consider the following nondegenerate diffusion process $(x^{\epsilon,1}(t), x^{\delta_2, \epsilon,2}(t), \ldots, x^{\delta_\ell, \epsilon,\ell}(t))$ satisfying
\begin{align}
\left.\begin{array}{l}
d x^{\epsilon,1}(t) = \hat{f}_1\bigl(t, x^{\epsilon,1}(t)\bigr) dt + \sqrt{\epsilon} \sigma\bigl(t, x^{\epsilon,1}(t))dW(t) \\
d x^{\delta_\ell,\epsilon,\ell}(t) = \hat{f}_{\ell}\bigl(t, x^{\epsilon,1}(t), x^{\delta_2, \epsilon,2}(t), \ldots, x^{\delta_\ell,\epsilon,\ell}(t)\bigr) dt + \sqrt{\delta_\ell} dW_\ell(t) dt, \\
\quad \quad\quad\quad\quad\quad\quad\quad \quad\quad\quad\quad\quad\quad\quad  \delta_\ell >0,  \,\,  s \le t \le T
\end{array}\right\} \label{Eq10}
\end{align}
for $\ell = 2, 3, \ldots, n$, with an initial condition
\begin{align}
  \bigl(x^{\epsilon,1}(s), x^{\delta_2, \epsilon,2}(s), \ldots, x^{\delta_\ell, \epsilon,\ell}(s)\bigr) = \bigl(x_s^{\epsilon,1}, x_s^{\epsilon,2}, \ldots, x_s^{\epsilon,\ell}\bigr).  \label{Eq11}
\end{align}
Moreover, $W_j(\cdot)$, for $j = 2, \ldots, \ell$, are $d$-dimensional standard Wiener processes (with $W_\ell(0)=0$) and independent to $W(\cdot)$.

Let $\tau_D^{\delta,\epsilon,\ell}=\tau_D^{\delta,\epsilon,{\ell}}(s, x^{\epsilon,1}, x^{\delta_2, \epsilon,2}, \ldots, x^{\delta_\ell,\epsilon,\ell})$ be the first exit-time for the diffusion process $x^{\delta_\ell, \epsilon,\ell}(t)$ (corresponding to the $\ell$th-subsystem, with a nondegenerate case) from the domain $D$. Later, we relate the exit probability of this diffusion process with that of the boundary value problem in Equation~\eqref{Eq9} as the limiting case, when $\delta_{\ell} \rightarrow 0$, for $\ell = 2,3, \dots, n$. 

Next, let us define the following
\begin{align*}
\begin{array}{c}
   \theta = \tau_D^{\epsilon,\ell} \wedge T, \quad\quad  \theta^{\delta} = \tau_D^{\delta,\epsilon,\ell} \wedge T, \\
   \bigl\Vert x^{\delta_\ell, \epsilon,\ell} - x^{\epsilon,\ell} \bigr\Vert_t = \sup\limits_{s \le r \le t} \Bigl\vert x^{\delta_\ell, \epsilon,\ell}(r) - x^{\epsilon,\ell}(r) \Bigr\vert,\quad \ell = 2, 3, \ldots, n.
\end{array}
\end{align*}

Then, we need the following lemma, which is useful for proving the above proposition.
\begin{lemma} \label{L1}
Suppose that $\epsilon > 0$ is fixed. Then, for any initial point $(x^{\epsilon,1}, \ldots, x^{\epsilon,\ell}) \in \mathbb{R}^{(\ell-1) \times d} \times D$, with $t > s$, the following statements hold true
\begin{enumerate} [(i)]
\item $\bigl\Vert x^{\delta_\ell, \epsilon,\ell} - x^{\epsilon,\ell} \bigr\Vert_t  \rightarrow 0$, 
\item $\theta^{\delta} \rightarrow \theta$, and 
\item $\bigl(x^{\delta_2,\epsilon,2}(\theta^{\delta}), \ldots, x^{\delta_\ell,\epsilon,\ell}(\theta^{\delta})\bigr) \rightarrow \bigl(x^{\epsilon,2}(\theta), \ldots, x^{\epsilon,\ell}(\theta)\bigr)$, 
\end{enumerate}
almost surely, as $\delta_\ell \rightarrow 0$, for $\ell = 2, 3, \ldots, n$.
\end{lemma}

\begin{proof} 
Part~(i): Note that, for a fixed $\epsilon > 0$ and $\ell \in \{ 2, 3, \ldots, n\}$, the following inequality holds 
\begin{align*}
\bigl\vert x^{\delta_\ell, \epsilon,\ell}(r) - x^{\epsilon,\ell}(r) \bigr\vert &\le  \int_{s}^t \Bigl \vert \hat{f}_{\ell}\bigl(r, x^{\epsilon,1}(r), \ldots, x^{\delta_\ell,\epsilon,\ell}(r)\bigr) - \hat{f}_{\ell}\bigl(r, x^{\epsilon,1}(r), \ldots, x^{,\epsilon,\ell}(r)\bigr) \Bigr\vert dr \notag \\
& \quad \quad\quad\quad\quad\quad + \sqrt{\delta_{\ell}} \vert W_{\ell}(t)\vert, \notag \\
& \le C \int_{s}^t \bigl \vert x^{\delta_\ell, \epsilon,\ell}(r) - x^{\epsilon,\ell}(r) \bigr\vert + \sqrt{\delta_{\ell}} \vert W_{\ell}(t)\vert,
\end{align*}
such that
\begin{align*}
\bigl \Vert x^{\delta_\ell, \epsilon,\ell}(r) - x^{\epsilon,\ell}(r) \bigr \Vert_t \le C \int_{s}^t \bigl \vert x^{\delta_\ell, \epsilon,\ell}(r) - x^{\epsilon,\ell}(r) \bigr\vert + \sqrt{\delta_{\ell}} \vert W_{\ell}(t)\vert,
\end{align*}
where $C$ is a Lipschitz constant. Using the Gronwall-Bellman inequality, we obtain the following
\begin{align*}
\bigl \Vert x^{\delta_\ell, \epsilon,\ell}(r) - x^{\epsilon,\ell}(r) \bigr \Vert_t \le c \sqrt{\delta_{\ell}} \bigl\Vert W_{\ell} \bigr\Vert_t,
\end{align*}
where $c$ is a constant that depends on $C$ and $(t-s)$. Hence, we have 
\begin{align*}
\bigl \Vert x^{\delta_\ell, \epsilon,\ell}(r) - x^{\epsilon,\ell}(r) \bigr \Vert_t  \rightarrow 0 \quad \text{as} \quad \delta_\ell \rightarrow 0,
\end{align*}
for each $\ell = 2, 3, \ldots, n$.

Part~(ii): Next, let us show $\theta$ satisfies the following bounds
\begin{align*}
\theta^{\ast} \le \theta \le \theta_{\ast},
\end{align*}
almost surely, where $\theta^{\ast} =\limsup_{\delta_\ell \rightarrow 0} \theta^{\delta}$ and $\theta_{\ast} =\liminf_{\delta_\ell \rightarrow 0} \theta^{\delta}$, for $\ell = 2, 3, \ldots, n$. 

Notice that $D$ is open, then it follows from Part~(i) that if $\theta = \tau_D^{\epsilon,\ell} \wedge T=T$, then $\theta^{\delta} = T$, almost surely, for all $\delta_{\ell}$ sufficiently small. Then, we will get Part~(ii). Similarly, if $\theta^{\delta} = T$ and $x^{\delta, \epsilon,\ell}(\theta^{\delta}) \in D$, then the statement in Part~(i) implies Part~(ii). Then, we can assume that $x^{\epsilon,\ell}(\theta) \in \partial D$ and $x^{\delta, \epsilon,\ell}(\theta^{\delta}) \in \partial D$. Moreover, if $x^{\delta, \epsilon,\ell}(\theta^{\delta}) \in \partial D$, then, from Part~(i), $x^{\epsilon,\ell}(\theta_{\ast}) \in \partial D$, almost surely, and, consequently, $\theta_{\ast} \ge \theta$, almost surely.

For the case $\theta^{\ast} \le \theta$, let us define an event $\Psi_{a,\alpha}$ (with $a >0$ and $\alpha > 0$) as follow: there exists $t \in [\theta, \theta + a]$ such that the distance $\varrho \bigl(x^{\epsilon,\ell}(t), D\cup \partial D \bigr) \ge \alpha$. Notice that if this holds together with $\bigl\Vert x^{\delta_\ell, \epsilon,\ell} - x^{\epsilon,\ell} \bigr\Vert_t < \alpha$, then we have $\theta^{\delta} < \theta + a$. Hence, from Part~(i), we have $\theta^{\ast} < \theta + a$ on $\Psi_{a,\alpha}$, almost surely.

On the other hand, from Assumption~\ref{AS1}(c), we have the following
\begin{align*}
\mathbb{P}_{s, x_{\widehat{1,\ell}}}^{\epsilon} \Bigl\{ \bigcup\nolimits_{\alpha >0} \Psi_{a,\alpha}\Bigr\} = 1.
\end{align*}
Then, 
\begin{align*}
\mathbb{P}_{s, x_{\widehat{1,\ell}}}^{\epsilon} \Bigl\{ \theta^{\ast} < \theta + a\Bigr\} = 1,
\end{align*}
since $a$ is an arbitrary, we obtain $\theta^{\ast} < \theta$, almost surely.

Finally, notice that the statement in Part~(iii) is a consequence of Part~(i) and Part~(ii). This completes the proof of Lemma~\ref{L1}. 
\end{proof}

\begin{proof} [Proof of Proposition~\ref{P1}]
Note that, from Assumption~\ref{AS1}(c), it is sufficient to show that $q^{\epsilon,\ell} (s, x^{\epsilon,1}, \ldots, x^{\epsilon,\ell})$ is a smooth solution (almost everywhere in $\Omega_{\ell}$ with respect to Lebesgue measure) to the boundary value problem in Equation~\eqref{Eq9}.

For a fixed $\ell \in \{ 2, 3, \ldots, n\}$, consider the following backward operator which corresponds to the nondegenerate diffusion process $(x^{\epsilon,1}(t), x^{\delta_2, \epsilon,2}(t), \ldots, x^{\delta_\ell, \epsilon,\ell}(t))$
\begin{align}
 \upsilon_s^{\delta,\epsilon,\ell} + \mathcal{L}^{\epsilon, \ell} \upsilon^{\delta,\epsilon,\ell} + \sum\nolimits_{j=2}^\ell \frac{\delta_j}{2} \triangle_{x^{\delta_j, \epsilon,j}} \upsilon^{\delta,\epsilon,\ell}= 0 \quad \text{in} \quad \Omega_{\ell} = (0, T) \times \mathbb{R}^{(\ell-1) \times d} \times D,  \label{Eq12}
\end{align}
where $\triangle_{x^{\delta_j, \epsilon,j}}$ is the Laplace operator in the variable $x^{\delta_j, \epsilon,j}$ and $\mathcal{L}^{\epsilon, \ell}$ is the backward operator in Equation~\eqref{Eq5}.

Next, define $\partial^{\ast} D$ as $\{T\} \times \mathbb{R}^{(\ell-1) \times d} \times D$, together with the boundary points $\bigl(s, x^{\epsilon,1}, x^{\delta_2, \epsilon,2},\\ \ldots, x^{\delta_\ell, \epsilon,\ell} \bigr) \in \Gamma_{\ell}^{+}$, with $0 < s <T$. Let $\psi \bigl(s, x^{\epsilon,1}, x^{\delta_2, \epsilon,2}, \ldots, x^{\delta_\ell, \epsilon,\ell} \bigr)$ be a function which is continuous on $\partial D$. Note that, from Assumption~\ref{AS1}(c), the backward operator in Equation~\eqref{Eq12} is uniformly parabolic and, therefore, its solution satisfies the following boundary condition
\begin{align}
 \upsilon^{\delta,\epsilon,\ell} \bigl(s, x^{\epsilon,1}, x^{\delta_2, \epsilon,2}, \ldots, x^{\delta_\ell, \epsilon,\ell} \bigr) = \psi \bigl(s, x^{\epsilon,1}, x^{\delta_2, \epsilon,2}, \ldots, x^{\delta_\ell, \epsilon,\ell} \bigr) \quad \text{on} \quad \partial^{\ast} D,  \label{Eq13}
\end{align}
where
\begin{align}
 \upsilon^{\delta,\epsilon,\ell} \bigl(s, x^{\epsilon,1},x^{\delta_2, \epsilon,2}, \ldots, x^{\delta_\ell, \epsilon,\ell} \bigr) = \mathbb{E}_{s, x_{\widehat{1,\ell}}}^{\delta, \epsilon} \Bigl\{ \psi \bigl(\theta^{\delta}, x^{\epsilon,1}, x^{\delta_2, \epsilon,2}, \ldots, x^{\delta_\ell, \epsilon,\ell} \bigr) \Bigr\},  \label{Eq14}
\end{align}
with $\theta^{\delta} = \tau_D^{\delta,\epsilon,\ell} \wedge T$.

In particular, let $\psi_k$, with $k=1, 2,\ldots$, be a sequence of bounded functions that are continuous on $\partial^{\ast} D$ and satisfying the following conditions
\begin{align*}
 \psi_k \bigl(s, x^{\epsilon,1}, x^{\delta_2, \epsilon,2}, \ldots, x^{\delta, \epsilon,\ell} \bigr) = \left\{\begin{array}{l l}
1  &\text{if} \, \bigl(s, x^{\epsilon,1}, x^{\delta_2, \epsilon,2}, \ldots, x^{\delta_\ell, \epsilon,\ell} \bigr) \in \Gamma_{T, \ell}^{+}\\
0 &\text{if} \,  \bigl(s, x^{\epsilon,1}, x^{\delta_2, \epsilon,2}, \ldots, x^{\delta_\ell, \epsilon,\ell} \bigr) \in \{T\} \times \mathbb{R}^{(\ell-1) \times d}  \\
&  \quad \quad\quad\quad\quad\quad\quad\quad \times D \quad \text{and} \quad \varrho\bigl(x^{\delta, \epsilon,\ell}, \partial D\bigr) > \frac{1}{k}
\end{array}\right. 
\end{align*}
and
\begin{align*}
0 \le \psi_k \bigl(s, x^{\epsilon,1}, x^{\delta_2, \epsilon,2}, \ldots, x^{\delta_\ell, \epsilon,\ell} \bigr) \le 1 \,\, & \text{if} \,\,  \bigl(s, x^{\epsilon,1}, x^{\delta_2, \epsilon,2}, \ldots, x^{\delta, \epsilon,\ell} \bigr) \in \{T\} \times \mathbb{R}^{(\ell-1) \times d}\\
 & \quad\quad\quad\quad\quad\quad  \times D \quad \text{and} \quad  \varrho\bigl(x^{\delta, \epsilon,\ell}, \partial D\bigr) \le \tfrac{1}{k}.
\end{align*}
Moreover, such bounded functions further satisfy the following
\begin{align}
\bigl\vert \psi_k - \psi_l \bigr\vert \rightarrow 0 \quad \text{as} \quad k,l \rightarrow \infty  \label{Eq15}
\end{align}
uniformly on any compact subset of $\bar{\Omega}_{\ell}$. Then, with $\psi = \psi_k$,
\begin{align*}
 \upsilon_k^{\delta,\epsilon,\ell} \bigl(s, x^{\epsilon,1}, x^{\delta_2, \epsilon,2}, \ldots, x^{\delta, \epsilon,\ell} \bigr) = \mathbb{E}_{s, x_{\widehat{1,\ell}}}^{\delta, \epsilon} \Bigl\{ \psi_k \bigl(\theta^{\delta}, x^{\epsilon,1}, x^{\delta_2, \epsilon,2}, \ldots, x^{\delta, \epsilon,\ell} \bigr) \Bigr\}
\end{align*}
satisfies Equation~\eqref{Eq13} and Equation~\eqref{Eq14}. Then, from the continuity of $\psi_k$ (cf. Lemma~\ref{L1}, Parts~(i)-(iii)) and the Lebesgue's dominated convergence theorem (see \cite[Chapter~4]{Roy88}), we have the following
\begin{align}
 \upsilon_k^{\delta,\epsilon,\ell} \bigl(s, x^{\epsilon,1}, x^{\delta_2, \epsilon,2}, \ldots, x^{\delta, \epsilon,\ell} \bigr) \rightarrow \underbrace{\mathbb{E}_{s, x_{\widehat{1,\ell}}}^{\delta, \epsilon} \Bigl\{ \psi_k \bigl(\theta^{\delta}, x^{\epsilon,1}, x^{\delta_2, \epsilon,2}, \ldots, x^{\delta, \epsilon,\ell} \bigr) \Bigr\}}_
{\triangleq  q_k^{\epsilon,\ell} \bigl(s, x^{\epsilon,1}, \ldots, x^{\epsilon,\ell} \bigr)} \label{Eq16}
\end{align}
as $\delta_\ell \rightarrow 0$, for $\ell = 2, 3, \ldots, n$, with $\theta = \tau_D^{\epsilon,\ell} \wedge T$. Furthermore, in the above equation, $\bigl(x^{\epsilon,1}(t), \ldots, x^{\epsilon,\ell}(t) \bigr)$ is a solution to Equation~\eqref{Eq10}, when $\delta_\ell=0$, for $\ell=2, 3, \ldots, n$, with an initial condition of Equation~\eqref{Eq11}.
 
Notice that $q_k^{\epsilon,\ell} \bigl(s, x^{\epsilon,1}, \ldots, x^{\epsilon,\ell} \bigr)$ satisfies the backward operator in Equation~\eqref{Eq12}, with $\upsilon^{\delta,\epsilon,\ell} =  \upsilon_k^{\delta,\epsilon,\ell}$, and, in addition, it is a distribution solution to the boundary value problem in Equation~\eqref{Eq9}, i.e.,
\begin{align*}
\int_{\Omega_{\ell}} \Bigl(-\phi_t + {\mathcal{L}^{\epsilon, \ell}}^{\ast} \phi \Bigr) q_k^{\epsilon,\ell} d \Omega_{\ell}, &= \lim_{\delta_j \rightarrow 0,\,\forall j \in \{2,\ldots,\ell\}}\int_{\Omega_{\ell}} \Bigl(-\phi_t + {\mathcal{L}^{\epsilon, \ell}}^{\ast} \phi \\
&\quad\quad\quad\quad \quad\quad\quad+ \sum\nolimits_{j=2}^\ell \frac{\delta_j}{2} \triangle_{x^{\delta_j, \epsilon,j}} \phi \Bigr) \upsilon_k^{\delta,\epsilon,\ell} d \Omega_{\ell}, \\
  &=0, 
\end{align*}
for any test function $\phi \in C_0^{\infty}(\Omega_{\ell})$. 

Finally, notice that 
\begin{align*}
q^{\epsilon,\ell} \bigl(s, x^{\epsilon,1}, \ldots, x^{\epsilon,\ell} \bigr) = \lim_{k \rightarrow \infty} q_k^{\epsilon,\ell} \bigl(s, x^{\epsilon,1}, \ldots, x^{\epsilon,\ell} \bigr),
\end{align*}
almost everywhere in $\Omega_{\ell}$. From Assumption~\ref{AS1}(b) (i.e., the hypoellipticity), $q^{\epsilon,\ell} \bigl(s, x^{\epsilon,1}, \ldots, \\ x^{\epsilon,\ell} \bigr)$ is a smooth solution to Equation~\eqref{Eq9} (almost everywhere) in $\Omega_{\ell}$ and continuous on the boundary of $\Omega_{0T,\ell}$. This completes the proof of Proposition~\ref{P1}. 
\end{proof}

\begin{remark} \label{R3}
Here, we remark that the statements in Proposition~\ref{P1} will make sense only if we require the following
\begin{align*}
\tau_D^{\epsilon,1}\, \ge\, \tau_D^{\epsilon,2}\, \ge \,\cdots \,\ge \, \tau_D^{\epsilon,\ell}, 
\end{align*}
where $\tau_D^{\epsilon,1} = \inf \bigl\{ t > s \, \bigl\vert \, x^{\epsilon,1}(t) \in \partial D \bigr\}$. It should be noted that such a condition, in general, depends on the constituting subsystems in Equation~\eqref{Eq1}, the admissible controls from the measurable sets $\prod_{j=1}^{\ell} \mathcal{U}_j$ and the given bounded open domain $D$.
\end{remark}

\subsection{Connection with control problems} \label{S2(2)}

\subsubsection{Deterministic minimum control problems} \label{S2(2a)}

Note that, from Proposition~\ref{P1}, the exit probability $q^{\epsilon,\ell} \bigl(s, x^{\epsilon,1}, \ldots, x^{\epsilon,\ell} \bigr)$ is a smooth solution to the boundary value problem in Equation~\eqref{Eq9}. Further, if we introduce the following logarithmic transformation (e.g., see \cite{Fle78} or \cite{EvaIsh85})
\begin{align}
 I^{\epsilon,\ell} \bigl(s, x^{\epsilon,1}, \ldots, x^{\epsilon,\ell} \bigr) = -\epsilon \log q^{\epsilon,\ell} \bigl(s, x^{\epsilon,1}, \ldots, x^{\epsilon,\ell} \bigr). \label{Eq17}
\end{align}
Then, the function $I^{\epsilon,\ell} \bigl(s, x^{\epsilon,1}, \ldots, x^{\epsilon,\ell} \bigr)$ satisfies the following boundary value problem
\begin{align}
\left.\begin{array}{c}
 I_s^{\epsilon,\ell} + \mathcal{L}^{\epsilon, \ell} I^{\epsilon,\ell} - \frac{1}{2} \bigl(I_{x^{\epsilon,1}}^{\epsilon,\ell}\bigr)^T a(s, x^{\epsilon,1}) I_{x^{\epsilon,1}}^{\epsilon,\ell} = 0 \quad \text{in} \quad \Omega_{\ell} = (0, T) \times \mathbb{R}^{(\ell-1) \times d} \times D \\
 I^{\epsilon,\ell} \bigl(s, x^{\epsilon,1}, \ldots, x^{\epsilon,\ell} \bigr) = 0 \quad \text{on} \quad \Gamma_{T, \ell}^{+} \\
 I^{\epsilon,\ell} \bigl(s, x^{\epsilon,1}, \ldots, x^{\epsilon,\ell} \bigr) = \infty \quad \text{on} \quad \{T\} \times \mathbb{R}^{(\ell-1) \times d} \times D
\end{array}\right\}  \label{Eq18}
\end{align}
where $\mathcal{L}^{\epsilon, \ell}$ is backward operator in Equation~\eqref{Eq5}. Observe that $I^{\epsilon,\ell} \bigl(s, x^{\epsilon,1}, \ldots, x^{\epsilon,\ell} \bigr)$ further satisfies the following dynamic programming equation
\begin{align}
0 =  I_s^{\epsilon,\ell} + \frac{\epsilon}{2} \operatorname{tr}\Bigl \{a\, I_{x^{\epsilon,1}x^{\epsilon,1}}^{\epsilon,\ell} \Bigr\} + \sum\nolimits_{j=1}^\ell \hat{f}_{j}^T\bigl(s, x^{\epsilon,1}, \ldots, x^{\epsilon,j} \bigr) I_{x^{\epsilon,j}}^{\epsilon,\ell} \notag \\ \quad\quad\quad + H^{\epsilon,\ell}\bigl(s, x^{\epsilon,1}, I_{x^{\epsilon,1}}^{\epsilon,\ell} \bigr) 
\quad \text{in} \quad \Omega_{\ell},  \label{Eq19}
\end{align}
for $\ell =2, 3, \ldots, n$, where
\begin{align}
H^{\epsilon,\ell}\bigl(s, x^{\epsilon,1}, p \bigr) = \hat{f}_{1}^T(s, x^{\epsilon,1}) p - \frac{1}{2} p^T \Bigl [a(s, x^{\epsilon,1})\Bigr]^{-1} p.  \label{Eq20}
\end{align}
Next, we define $L^{\epsilon,\ell}\bigl(s, x^{\epsilon,1}, \hat{u} \bigr)$ as
\begin{align}
L^{\epsilon,\ell}\bigl(s, x^{\epsilon,1}, \hat{u} \bigr) = \frac{1}{2} \Bigl(\hat{f}_{1}(s, x^{\epsilon,1}) - \hat{u} \Bigr)^T \Bigl [a(s, x^{\epsilon,1})\Bigr]^{-1} \Bigl(\hat{f}_{1}(s, x^{\epsilon,1}) - \hat{u} \Bigr).  \label{Eq21}
\end{align}
Then, we observe that there is a duality between $H^{\epsilon,\ell}\bigl(s, x^{\epsilon,1}, p \bigr)$ and $L^{\epsilon,\ell}\bigl(s, x^{\epsilon,1}, \hat{u} \bigr)$ such that
\begin{align}
L^{\epsilon,\ell}\bigl(s, x^{\epsilon,1}, \hat{u} \bigr) = \sup_{p} \biggl \{ H^{\epsilon,\ell}\bigl(s, x^{\epsilon,1}, p \bigr) - p^T \hat{u} \biggr\}  \label{Eq22}
\end{align}
and
\begin{align}
H^{\epsilon,\ell}\bigl(s, x^{\epsilon,1}, p \bigr) = \inf_{\hat{u}} \biggl \{ L\bigl(s, x^{\epsilon,1}, \hat{u} \bigr) + p^T \hat{u} \biggr\}.  \label{Eq23}
\end{align}
Furthermore, if we set $\epsilon=0$ in Equation~\eqref{Eq19}, then we have the following dynamic programming equation (e.g., see \cite[Chapter~4]{FemRi75})
\begin{align}
I_s^{0,\ell} + \sum\nolimits_{j=1}^\ell \hat{f}_{j}^T\bigl(s, x^{\epsilon,1}, \ldots, x^{\epsilon,j} \bigr) I_{x^{\epsilon,j}}^{0,\ell} +  \inf_{\hat{u}} \biggl \{ L^{\epsilon,\ell}\bigl(s, x^{\epsilon,1}, \hat{u} \bigr) + \bigl(I_{x^{\epsilon,1}}^{0,\ell} \bigr)^T \hat{u} \biggr\} = 0,  \label{Eq24}
\end{align}
for a family of deterministic minimum control problems corresponding to the following system of equations
\begin{align}
\left.\begin{array}{l}
d x^{0,1}(t) = \hat{u}(t) dt \\
d x^{0,\ell}(t) = \hat{f}_{\ell} \bigl(t, x^{0,1}(t), \ldots, x^{0,\ell}(t)\bigr) dt, \quad  s \le t \le T
\end{array}\right\}  \label{Eq25} 
\end{align}
for $\ell = 2, 3, \ldots, n$, with an initial condition
\begin{align*}
  \bigl(x^{0,1}(s), \ldots, x^{0,\ell}(s)\bigr) =  \bigl(x_s^{\epsilon,1}, \ldots, x_s^{\epsilon,\ell}\bigr)
\end{align*}
and the associated value functions
\begin{align}
I^{0,\ell} \bigl(s, x^{\epsilon,1}, \ldots, x^{\epsilon,\ell} \bigr) = \inf_{\hat{u} \in \hat{U}\bigl(s, x_s^{\epsilon,1}, \ldots, x_s^{\epsilon,\ell} \bigr)} \int_{s}^{\theta} L^{\epsilon,\ell}\bigl(t, x^{0,1}(t), \hat{u}(t) \bigr) dt, \label{Eq26}
\end{align}
where $\theta$ is the exit-time for $x^{0,\ell}(t)$ from the domain $D$, $\hat{U}\bigl(s, x^{\epsilon,1}, \ldots, x^{\epsilon,\ell} \bigr)$ is a class of continuous functions for which $\theta \le T$, and $\bigl(\theta, x^{0,1}(\theta), \ldots, x^{0,\ell}(\theta) \bigr)\in  \Gamma_{T, \ell}^{+}$.  

In the following subsection, using ideas from stochastic control theory (see \cite{Fle78} for similar ideas), we present results useful for proving the following asymptotic property 
\begin{align}
 I^{\epsilon,\ell} \bigl(s, x^{\epsilon,1}, \ldots, x^{\epsilon,\ell} \bigr) \rightarrow I^{0,\ell} \bigl(s, x^{\epsilon,1}, \ldots, x^{\epsilon,\ell} \bigr) \quad \text{as} \quad \epsilon \rightarrow 0, \label{Eq27}
\end{align}
for each $\ell=2, 3, \ldots, n$. The starting point for such an analysis is to introduce a family of related stochastic control problems whose dynamic programming equation, for $\epsilon > 0$, is given by Equation~\eqref{Eq19}. Then, this further allows us to reinterpret the exit probability function as a value function for a family of stochastic control problems that are associated with the underlying chain of distributed control systems.

\subsubsection{Stochastic control problems} \label{S2(2b)}
Consider the following boundary value problem
\begin{align}
\left.\begin{array}{c}
g_s^{\epsilon,\ell} + \frac{\epsilon}{2} \operatorname{tr}\Bigl \{a\, g_{x^{\epsilon,1}x^{\epsilon,1}}^{\epsilon,\ell} \Bigr\} + \sum\nolimits_{j=1}^\ell \hat{f}_{j}^T g_{x^{\epsilon,j}}^{\epsilon,\ell} = 0 \quad \text{in} \quad \Omega_{\ell} \\
g^{\epsilon,\ell}\bigl(s, x^{\epsilon,1}, \ldots, x^{\epsilon,\ell} \bigr) = \mathbb{E}_{s, x_{\widehat{1,\ell}}}^{\epsilon} \Bigl \{ \exp\Bigl(-\frac{1} {\epsilon} \Phi^{\epsilon,\ell} \bigl(s, x^{\epsilon,2}, \ldots, x^{\epsilon,\ell} \bigr) \Bigr) \Bigr\} \quad \text{on} \quad \partial^{\ast} \Omega_{\ell}
\end{array}\right\}  \label{Eq28} 
\end{align}
where the function $\Phi^{\epsilon,\ell} \bigl(s, x^{\epsilon,2}, \ldots, x^{\epsilon,\ell} \bigr)$ is a bounded, nonnegative Lipschitz function such that
\begin{align}
\Phi^{\epsilon,\ell} \bigl(s, x^{\epsilon,2}, \ldots, x^{\epsilon,\ell} \bigr) =0, \quad \forall \bigl(s, x^{\epsilon,1}, x^{\epsilon,2}, \ldots, x^{\epsilon,\ell} \bigr) \in \Gamma_{T, \ell}^{+}. \label{Eq29} 
\end{align}
Observe that the function $g_s^{\epsilon,\ell} \bigl(s, x^{\epsilon,1}, \ldots, x^{\epsilon,\ell} \bigr)$ is a smooth solution in $\Omega_{\ell}$ to the backward operator in Equation~\eqref{Eq5}; and it is continuous on $\partial^{\ast} \Omega_{\ell}$. Moreover, if we introduce the following logarithm transformation (cf. Equation~\eqref{Eq17})
\begin{align}
  J^{\epsilon, \ell} \bigl(s, x^{\epsilon,1}, \ldots, x^{\epsilon,\ell} \bigr) = - \epsilon \log g_s^{\epsilon,\ell}\bigl(s, x^{\epsilon,1}, \ldots, x^{\epsilon,\ell} \bigr). \label{Eq30} 
\end{align}
Then,  $J^{\epsilon, \ell}\bigl(s, x^{\epsilon,1}, \ldots, x^{\epsilon,\ell}\bigr)$ satisfies the following 
\begin{align}
0 =  J_s^{\epsilon,\ell} + \frac{\epsilon}{2} \operatorname{tr}\Bigl \{a\, J_{x^{\epsilon,1}x^{\epsilon,1}}^{\epsilon,\ell} \Bigr\} + \sum\nolimits_{j=1}^\ell \hat{f}_{j}^T J_{x^{\epsilon,j}}^{\epsilon,\ell} + H^{\epsilon,\ell}\bigl(s, x^{\epsilon,1}, J_{x^{\epsilon,1}}^{\epsilon,\ell} \bigr) ~ \text{in} ~ \Omega_{\ell},  \label{Eq31}
\end{align}
for $\ell =2, 3, \ldots, n$, where
\begin{align}
H^{\epsilon,\ell}\bigl(s, x^{\epsilon,1}, J_{x^{\epsilon,1}}^{\epsilon,\ell} \bigr) = \hat{f}_{1}^T(s, x^{\epsilon,1}) J_{x^{\epsilon,1}}^{\epsilon,\ell} - \frac{1}{2} \Bigl(J_{x^{\epsilon,1}}^{\epsilon,\ell}\Bigr)^T \Bigl [a(s, x^{\epsilon,1})\Bigr]^{-1} J_{x^{\epsilon,1}}^{\epsilon,\ell}.  \label{Eq32}
\end{align}
Note that the duality relation between $H^{\epsilon,\ell} \bigl(s, x^{\epsilon,1},\,\cdot\bigr)$ and $L^{\epsilon,\ell} \bigl(s, x^{\epsilon,1},\,\cdot\bigr)$, i.e.,
\begin{align}
H^{\epsilon,\ell} \bigl(s, x^{\epsilon,1}, J_{x^{\epsilon,1}}^{\epsilon,\ell}\bigr) = \inf_{\hat{u}} \biggl \{L^{\epsilon,\ell} \bigl(s, x^{\epsilon,1}, \hat{u} \bigr) + \Bigl(J_{x^{\epsilon,1}}^{\epsilon,\ell}\Bigr)^T \hat{u} \biggr\}, \label{Eq33}
\end{align}
where
\begin{align*}
L^{\epsilon,\ell} \bigl(s, x^{\epsilon,1}, \hat{u} \bigr) = \frac{1}{2} \Bigl(\hat{f}_{1}(s, x^{\epsilon,1}) - \hat{u} \Bigr)^T \Bigl [a(s, x^{\epsilon,1})\Bigr]^{-1} \Bigl(\hat{f}_{1}(s, x^{\epsilon,1}) - \hat{u} \Bigr).
\end{align*}
Then, it is easy to see that $J^{\epsilon,\ell}\bigl(s, x^{\epsilon,1}, \ldots, x^{\epsilon,\ell} \bigr)$ is a solution in  $\Omega_{\ell}$, with $J^{\epsilon, \ell}=\Phi^{\epsilon,\ell}$ on $\partial^{\ast} \Omega_{\ell}$, to the dynamic programming in Equation~\eqref{Eq31}, where the latter is associated with the following stochastic control problem 
\begin{align}
 J^{\epsilon, \ell}\bigl(s, x^{\epsilon,1}, \ldots, x^{\epsilon,\ell} \bigr) = \inf_{\hat{u} \in \hat{U}\bigl(s, x_s^{\epsilon,1}, \ldots, x_s^{\epsilon,\ell} \bigr)} \mathbb{E}_{s, x_{\widehat{1,\ell}}}^{\epsilon}\Biggl\{  \int_{s}^{\theta}  L^{\epsilon,\ell} \bigl(s,x^{\epsilon,1}, \hat{u} \bigr)dt \notag \\
 + \Phi^{\epsilon,\ell} \bigl(\theta, x^{\epsilon,2}, \ldots, x^{\epsilon,\ell} \bigr) \Biggr\}, \label{Eq34}
\end{align}
that corresponds to system of stochastic differential equations
\begin{align}
\left.\begin{array}{l}
d x^{\epsilon,1}(t) = \hat{u}(t) dt + \sqrt{\epsilon} \,\sigma \bigl(t, x^{\epsilon,1}(t)\bigr) dW(t) \\
d x^{\epsilon,\ell}(t) = \hat{f}_{\ell} \bigl(t, x^{\epsilon,1}(t), \ldots, x^{\epsilon,\ell}(t)\bigr) dt, \quad  s \le t \le T
\end{array}\right\}  \label{Eq35} 
\end{align}
for $\ell =2, 3, \ldots, n$, with an initial condition
\begin{align*}
  \bigl(x^{\epsilon,1}(s), \ldots, x^{\epsilon,\ell}(s)\bigr) =  \bigl(x_s^{\epsilon,1}, \ldots, x_s^{\epsilon,\ell}\bigr),
\end{align*}
and where $\hat{U}\bigl(s, x^{\epsilon,1}, \ldots, x^{\epsilon,\ell} \bigr)$ is a class of continuous functions for which $\theta \le T$ ~ and ~ $\bigl(\theta, x^{\epsilon,1}(\theta), \ldots, x^{\epsilon,\ell}(\theta) \bigr)\in \Gamma_{T, \ell}^{+}$. 

In what follows, we provide bounds (i.e., the asymptotic lower/upper bounds) on the exit probability $q^{\epsilon,\ell} \bigl(s, x^{\epsilon,1}, \ldots, x^{\epsilon,\ell} \bigr)$ for each $\ell =2, 3, \ldots, n$.

Define
\begin{align}
 I_{D}^{\epsilon,\ell} \Bigl(\bigl(s, x^{\epsilon,1}, \ldots, x^{\epsilon,\ell} \bigr); \, \partial D \Bigr) &= -\lim_{\epsilon \rightarrow 0} \epsilon\,\log \mathbb{P}_{s, x_{\widehat{1,\ell}}}^{\epsilon} \Bigl \{ x^{\epsilon,\ell}(\theta) \in \partial D \Bigl\}, \notag \\
 & \triangleq -\lim_{\epsilon \rightarrow 0} \epsilon\,\log q^{\epsilon,\ell} \bigl(s, x^{\epsilon,1}, \ldots, x^{\epsilon,\ell} \bigr), \label{Eq36}
\end{align}
where $\theta$ (or $\theta=\tau_D^{\epsilon,\ell} \wedge T$) is the first exit-time of $x^{\epsilon,\ell}(t)$ from the domain $D$. Further, let us introduce the following supplementary minimization problem 
\begin{align}
\tilde{I}_{D}^{\epsilon,\ell} \Bigl(s, \varphi, \theta\Bigr) = \inf_{\varphi \in C_{sT}\bigl([s, T], \mathbb{R}^d\bigr), \theta \ge s} \int_{s}^{\theta} L^{\epsilon,\ell}\big(t, \varphi(t),\dot{\varphi}(t)\big)dt, \label{Eq37}
\end{align}
where the infimum is taken among all $\varphi(\cdot) \in C_{sT}\bigl([s, T], \mathbb{R}^d\bigr)$ (i.e., from the space of $\mathbb{R}^d$-valued (locally) absolutely continuous functions, with $\int_{s}^T \bigl\vert \dot{\varphi}(t) \bigr\vert^2 dt < \infty$ for each $T > s$) and $\theta \ge s > 0$ such that $\varphi(s)=x_s^{\epsilon,1}$, $\bigl(t, \varphi(t), x^{\epsilon,2}(t), \ldots, x^{\epsilon,\ell}(t) \bigr) \in \Omega_{\ell}$, for all $t \in [s,\,\theta)$, and $\bigl(\theta, \varphi(\theta),x^{\epsilon,2}(\theta), \ldots, x^{\epsilon,\ell}(\theta) \bigr) \in \Gamma_{T, \ell}^{+}$. Then, it is easy to see that
\begin{align}
 \tilde{I}_{D}^{\epsilon,\ell} \Bigl(s, \varphi, \theta\Bigr) = I_{D}^{\epsilon,\ell} \Bigl(\bigl(s, x^{\epsilon,1}, \ldots, x^{\epsilon,\ell} \bigr); \,\partial D \Bigr). \label{Eq38}
\end{align}

Next, we state the following lemma that will be useful for proving Proposition~\ref{P2} (cf. \cite[Lemma~3.1]{Fle78}).
\begin{lemma} \label{L2} 
If $\varphi(\cdot) \in C_{sT}\bigl([s, T], \mathbb{R}^d\bigr)$, for $s> 0$, and $\varphi(s)=x_s^{\epsilon,1}$, $\bigl(t, \varphi(t), x^{\epsilon,2}(t), \ldots, \\ x^{\epsilon,\ell}(t) \bigr) \in \Omega_{\ell}$, for all $t \in [s,T)$, then $\lim_{T \rightarrow \infty} \int_{s}^{T} L^{\epsilon,\ell}\big(t,\varphi(t),\dot{\varphi}(t)\big)dt = +\infty$.
\end{lemma}

Consider again the stochastic control problem in Equation~\eqref{Eq34} (together with Equation~\eqref{Eq35}). Suppose that  $\Phi_M^{\epsilon,\ell}$ (with $\Phi_M^{\epsilon,\ell} \ge 0$) is class $C^2$ such that $\Phi_M^{\epsilon,\ell} \rightarrow +\infty$ as $M \rightarrow \infty$ uniformly on any compact subset of $\Omega_{\ell} \setminus \bar{\Gamma}_{T, \ell}^{+}$ and $\Phi_M^{\epsilon,\ell}$ on $\Gamma_{T, \ell}^{+}$. Further, if we let $J^{\epsilon,\ell} = J_{\Phi_M}^{\epsilon,\ell}$, when  $\Phi^{\epsilon,\ell} = \Phi_M^{\epsilon,\ell}$, then we have the following lemma.
\begin{lemma} \label{L3}
Suppose that Lemma~\ref{L2} holds, then we have
\begin{align}
 \liminf_{\substack{M \rightarrow \infty \\ \bigl(t, x^{\epsilon,1}(t), \ldots, x^{\epsilon,\ell}(t) \bigr) \rightarrow \bigl(s, x^{\epsilon,1}(s), \ldots, x^{\epsilon,\ell}(s) \bigr)}} J_{\Phi_M}^{\epsilon,\ell}(\bigl(s, x^{\epsilon,1}, \ldots, x^{\epsilon,\ell} \bigr)) \ge I^{\epsilon,\ell} \bigl(s, x^{\epsilon,1}, \ldots, x^{\epsilon,\ell} \bigr).  \label{Eq39}
\end{align}
\end{lemma}
Then, we have the following result.
\begin{proposition} \label{P2}
Suppose that Lemma~\ref{L2} holds, then we have
\begin{align}
  I^{\epsilon,\ell} \bigl(s, x^{\epsilon,1}, \ldots, x^{\epsilon,\ell} \bigr) \rightarrow I^{0,\ell} \bigl(s, x^{\epsilon,1}, \ldots, x^{\epsilon,\ell} \bigr) \quad \text{as} \quad \epsilon \rightarrow 0, \label{Eq40}
\end{align}
uniformly for all $\bigl(s, x^{\epsilon,1}(s), \ldots, x^{\epsilon,\ell}(s) \bigr)$ in any compact subset $\bar{\Omega}_{\ell}$. 
\end{proposition}

\begin{proof}
It is suffices to show the following conditions
\begin{align}
  \limsup_{\epsilon \rightarrow 0} \epsilon\,\log \mathbb{P}_{s, x_{\widehat{1,\ell}}}^{\epsilon} \Bigl \{ x^{\epsilon,\ell}(\theta) \in \partial D \Bigl\} \le -I_{D}^{\epsilon,\ell} \Bigl(\bigl(s, x^{\epsilon,1}, \ldots, x^{\epsilon,\ell} \bigr); \, \partial D \Bigr) \label{Eq41}
\end{align}
and
\begin{align}
 \liminf_{\epsilon \rightarrow 0} \epsilon\,\log \mathbb{P}_{s, x_{\widehat{1,\ell}}}^{\epsilon} \Bigl \{ x^{\epsilon,\ell}(\theta) \in \partial D \Bigl\} \ge -I_{D}^{\epsilon,\ell} \Bigl(\bigl(s, x^{\epsilon,1}, \ldots, x^{\epsilon,\ell} \bigr); \, \partial D \Bigr), \label{Eq42}
\end{align}
uniformly for all $\bigl(s, x^{\epsilon,1}(s), \ldots, x^{\epsilon,\ell}(s) \bigr)$ in any compact subset $\bar{\Omega}_{\ell}$.

Note that $I_{D}^{\epsilon,\ell} \Bigl(\bigl(s, x^{\epsilon,1}, \ldots, x^{\epsilon,\ell} \bigr); \, \partial D \Bigr) = I^{\epsilon,\ell} \bigl(s, x^{\epsilon,1}, \ldots, x^{\epsilon,\ell} \bigr)$ (cf. Equation~\eqref{Eq38}), then the upper bound in Equation~\eqref{Eq41} can be verified using the Ventcel-Freidlin asymptotic estimates (see \cite[pp.\,332--334]{Fre76}, \cite{VenFre70} or \cite{Ven73}). 

On the other hand, to prove the lower bound in Equation~\eqref{Eq42}, we introduce a penalty function $\Phi_M^{\epsilon,\ell} \bigl(\cdot\bigr)$ (with $\Phi_M^{\epsilon,\ell}\bigl(t, y^{1}, \ldots, y^{\ell} \bigr)=0$ for $\bigl(t, y^{1}, \ldots, y^{\ell}\bigr) \in \Gamma_{T, \ell}^{+}$); and write $g_s^{\epsilon,\ell}\bigl(\cdot\bigr)=g_{s,M}^{\epsilon,\ell}\bigl(\cdot\bigr) \Bigl(\equiv \mathbb{E}_{s, x_{\widehat{1,\ell}}}^{\epsilon} \Bigl \{\exp\Bigl(-\frac{1} {\epsilon} \Phi_{M}^{\epsilon,\ell} \bigl(\cdot\bigr) \Bigr) \Bigr\}\Bigr)$ and $J^{\epsilon, \ell}=J_{\Phi_M}^{\epsilon,\ell}(\cdot)$, with $\Phi^{\epsilon,\ell} \bigl(\cdot\bigr)=\Phi_M^{\epsilon,\ell}\bigl(\cdot\bigr)$. From the boundary condition in Equation~\eqref{Eq28}, then, for each $M$, we have
\begin{align}
 g^{\epsilon,\ell}\bigl(s, x^{\epsilon,1}, \ldots, x^{\epsilon,\ell} \bigr) \le g_{s,M}^{\epsilon,\ell}\bigl(s, x^{\epsilon,1}, \ldots, x^{\epsilon,\ell} \bigr). \label{Eq43}
\end{align}
Using Lemma~\ref{L3} and noting further the following 
\begin{align}
J_{\Phi_M}^{\epsilon,\ell} \bigl(s, x^{\epsilon,1}, \ldots, x^{\epsilon,\ell} \bigr) \ge I_{D}^{\epsilon,\ell} \Bigl(\bigl(s, x^{\epsilon,1}, \ldots, x^{\epsilon,\ell} \bigr); \, \partial D \Bigr).  \label{Eq44}
\end{align}
Then, the lower bound in Equation~\eqref{Eq42} holds uniformly for all $\bigl(s, x^{\epsilon,1}(s), \ldots, x^{\epsilon,\ell}(s) \bigr)$ in any compact subset $\bar{\Omega}_{\ell}$. This completes the proof of Proposition~\ref{P2}. 
\end{proof}

\begin{remark} \label{R4}
Here, it is worth remarking that Proposition~\ref{P2} is useful for obtaining an asymptotic information on the behavior of the distributed control systems. For example, an asymptotic information on the time-duration for which the diffusion process $x^{\epsilon,i}(t)$ is confined to the given or prescribed domain $D$ (with respect to the admissible controls $u_{i} = \kappa_{i}(x^{\epsilon,i}) \in \mathcal{U}_{i}$, for $i = 1, 2 \ldots, n$).
\end{remark}

\section{Concluding Remarks} \label{S3}
In this paper, we have provided an asymptotic estimate on the exit probability with which the diffusion process (corresponding to a chain of distributed control systems with small random perturbation) exits from the given bounded open domain during a certain time interval. In particular, we have argued that such an asymptotic estimate can be obtained based on a precise interpretation of the exit probability function as a value function for a family of stochastic control problems that are associated with the underlying chain of distributed control systems. Finally, it is worth mentioning that it would be interesting to characterize, in line with \cite{DelM10}, how the random perturbation propagates through the chain of distributed control systems for a fixed perturbation parameter $\epsilon>0$.

%
%
%
%
%

\end{document}